\documentclass[11pt,leqno]{article}
\usepackage{mathrsfs,amssymb,amsmath,amsthm, color}
\usepackage[dvips]{graphicx}

\makeatletter
	
	\@addtoreset{equation}{section}
\makeatother

\begin{document}

\renewcommand{\theenumi}{\rm (\roman{enumi})}
\renewcommand{\labelenumi}{\rm \theenumi}

\newtheorem{thm}{Theorem}[section]
\newtheorem{defi}[thm]{Definition}
\newtheorem{lem}[thm]{Lemma}
\newtheorem{prop}[thm]{Proposition}
\newtheorem{cor}[thm]{Corollary}
\newtheorem{exam}[thm]{Example}
\newtheorem{conj}[thm]{Conjecture}
\newtheorem{rem}[thm]{Remark}
\allowdisplaybreaks

\title{H\"older and Lipschitz continuity of the solutions to parabolic equations of the non-divergence type}
\author{Seiichiro Kusuoka\footnote{e-mail: kusuoka@okayama-u.ac.jp}
\vspace{0.1in}\\
\normalsize Research Institute for Interdisciplinary Science, Okayama University\\
\normalsize 3-1-1 Tsushima-naka, Kita-ku Okayama 700-8530, Japan}
\maketitle

\begin{abstract}
We consider time-inhomogeneous, second order linear parabolic partial differential equations of the non-divergence type, and assume the ellipticity and the continuity on the coefficient of the second order derivatives and the boundedness on all coefficients.
Under the assumptions we show the H\"older continuity of the solution in the spatial component.
Furthermore, additionally assuming the Dini continuity of the coefficient of the second order derivative, we have the better continuity of the solution.
In the proof, we use a probabilistic method, in particular the coupling method.
As a corollary, under an additional assumption we obtain the H\"older and Lipschitz continuity of the fundamental solution in the spatial component.
\end{abstract}

{\bf 2010 AMS Classification Numbers:} 35B65, 35K10, 60H10, 60H30, 60J60.

 \vskip0.2cm


\section{Introduction and main result}\label{section:intro}

Let $a(t,x)=(a_{ij}(t,x))$ be a symmetric $d\times d$-matrix-valued bounded measurable function on $[0,\infty ) \times {\mathbb R}^d$ which is uniformly positive definite, i.e.
\begin{equation}\label{ass:a1}
\Lambda ^{-1} I \leq a(t,x) \leq \Lambda I
\end{equation}
where $\Lambda$ is a positive constant and $I$ is the unit matrix.
Moreover, we assume the continuity of $a(t,\cdot )$ uniformly in $t$, i.e. for $R>0$ there exists a continuous and nondecreasing function $\rho _R$ on $[0,\infty )$ such that $\rho _R (0) =0$ and
\begin{equation}\label{ass:a2}
\sup _{t\in [0,\infty)} \max _{i,j=1,2,\dots, d} \left| a_{ij}(t,x) - a_{ij}(t,y)\right| \leq \rho _R (|x-y|) ,\quad x,y\in B(0;R).
\end{equation}
Let $b(t,x)=(b_i(t,x))$ be an ${\mathbb R}^d$-valued bounded measurable function on $[0,\infty ) \times {\mathbb R}^d$, and $c(t,x)$ be a bounded measurable function on $[0,\infty ) \times {\mathbb R}^d$.
Consider the following parabolic partial differential equation of the non-divergence type:
\begin{equation}\label{PDE}
\left\{ \begin{array}{rl}
\displaystyle \frac{\partial}{\partial t} u(t,x) & \displaystyle = \frac 12 \sum _{i,j=1}^d a_{ij}(t,x)\frac{\partial ^2}{\partial x_i \partial x_j} u(t,x) +  \sum _{i=1}^d b_i(t,x) \frac{\partial }{\partial x_i} u(t,x) + c(t,x) u(t,x) \!\!\!\!\!\!\!\!\!\! \\[3mm]
\displaystyle u(0,x)& =f(x).\displaystyle 
\end{array}\right.
\end{equation}
The existence and the uniqueness of the mild solution to (\ref{PDE}) are obtained under the assumptions (\ref{ass:a1}) and (\ref{ass:a2}) (see \cite{StVa}).
In the present paper, we consider the regularity in the spatial component of the solution and the fundamental solution to (\ref{PDE}).

It is well-known that the solution and the fundamental solution has the regularity according to the regularity of the coefficients $a$, $b$ and $c$.
When $a$, $b$ and $c$ are not sufficiently smooth, the argument to obtain the regularity of the solution is completely different from the case of the sufficiently smooth coefficients.
In the case that $a$, $b$ and $c$ are H\"older continuous, parametrix method is the standard way to see the regularity of the solutions and the fundamental solutions (see \cite{Fr} and \cite{LSU}).
The method enables us to construct the fundamental solution directly.
Furthermore, in the case of H\"older continuous coefficients, the Schauder estimate is known for the solutions to the parabolic equations, and as the consequence, we have  $u(t,\cdot) \in C^2({\mathbb R}^d)$ (see e.g. \cite{Kr}).

When all of the coefficients are independent of the time component and continuous in the spatial component, the method of analytic semigroups are available (see \cite{Ste} and \cite{Lu}).
As the result, we obtain that the solution belongs to $C^{1}\cap W_{\rm loc}^{2,p}$ for $p>d$ (see Theorem 6 in \cite{Ste}).
The continuities of the coefficients of the lower-order derivative terms are assumed in Theorem 6 in \cite{Ste} and Chapter 5 in \cite{Lu}.
However, the continuities can be removed (see the remarks in \cite{Ste} and the footnote at page 69 in \cite{Lu}).
The method is applicable to the case of time-dependent coefficients.
However, for the application we need that all of the coefficients belong to $C^{0,\alpha}([0,\infty)\times {\mathbb R}^d)$ for $\alpha >0$ (the space of the functions which are H\"older continuous in the spatial component uniformly in the time component).
The reason is that when we apply the method to the case of time-dependent coefficients, we have to estimate the variation of the coefficients in time (see Theorem 5.1.16 in \cite{Lu} and Theorem 7 in \cite{Ste}).

The case that $a$ is uniformly positively definite and bounded, $a(t,\cdot )$ is continuous uniformly in $t$, $b$ is bounded and measurable, and $c=0$, is studied by Stroock and Varadhan \cite{StVa1,StVa2} (the results are summarized in \cite{StVa}).
Under the setting, they obtained the uniqueness of the mild solution.
It is also known that; if we remove the continuity of $a$, then the uniqueness does not hold (see \cite{Sa}).
Moreover, Stroock and Varadhan obtained the existence of the fundamental solution $p(0,x;t,y)$ for almost every $t$.
On the other hand, even if $a$ is uniformly continuous and $b=0$, there is an example that the fundamental solution does not exist for a certain $t$ (see \cite{FaKe}).
We remark that the case that $c$ is bounded measurable is reduced to the case that $c=0$ by the Feynman-Kac formula.

The equation in which we are interested in the present paper is of the non-divergence type, but we also comment on the case of the equation of the divergence type, i.e. the case that in (\ref{PDE}) the term $\frac 12 \sum _{i,j=1}^d a_{ij}(t,x)\frac{\partial ^2}{\partial x_i \partial x_j}$ is replaced by $\frac 12 \sum _{i,j=1}^d \frac{\partial}{\partial x_i} \left( a_{ij}(t,x) \frac{\partial}{\partial x_j} \right)$.
In the case of the equation of the divergence type, we can apply the variational method, and many results have been obtained.
Nash \cite{Na} and Di Giorgi \cite{DG} independently proved that; when $a$ is uniformly positively definite and bounded, $b=0$ and $c=0$, then the solution exists and is H\"older continuous.
After that, Moser \cite{Mo} showed the Harnack inequality to the solution, and obtained the H\"older continuity of the solution as the consequence of the Harnack inequality. Later, Aronson \cite{Ar} generalized the results to the case that $b$ and $c$ are bounded measurable and obtained the Gaussian bounds of the fundamental solution.
These results are summarized in \cite{St}.

In the present paper, we consider the H\"older continuity and the Lipschitz continuity in $x$ of the solution $u(t,x)$ to (\ref{PDE}) under the assumptions (\ref{ass:a1}) and the continuity of $a(t,\cdot)$ uniformly in $t$.
As mentioned above, the uniqueness of the mild solution is obtained from Stroock and Varadhan's result (see \cite{StVa}) and the Feynman-Kac formula.
We prepare the Markov process $X$ associated with the parabolic equation which consists of the coefficients $a$ and $b$ (see (\ref{SDEX})), and obtain the H\"older continuity of $u(t,\cdot )$ with a constant depending on the transition probability measure $p^X$ of $X$ (see \ref{thm:main1} of Theorem \ref{thm:main}).
In the case that $a(t,\cdot)$ is locally Dini continuous uniformly in $t$ (see Definition \ref{def:conti}) and in the case that $a(t,\cdot)$ Dini continuous uniformly in $t$ (see Definition \ref{def:conti}),  then we have the better continuity of $u(t,\cdot )$ (see \ref{thm:main2} and \ref{thm:main3} of Theorem \ref{thm:main}, respectively).
The parabolic equations of the non-divergence type with Dini continuous coefficients are studied in \cite{PrEi}.
Under a little stronger assumption on $a$ about the continuity, and the Dini continuity of $b(t,\cdot)$ and $c(t,\cdot)$ uniformly in $t$, the bounds of the derivatives of the fundamental solution have been obtained in Theorem 19 of \cite{PrEi}.
We remark that the assumptions in the present paper are weaker than that in Theorem 19 of \cite{PrEi}.
The Dini continuity is also appears in \cite{CaKe}.
The equation concerned in \cite{CaKe} is of the divergence type, however a similar result is obtained (see Corollary 1.2.22 in \cite{CaKe}).

In the main theorem (Theorem \ref{thm:main}), the constant which appears in the H\"older continuity is depending on the transition probability measure $p^X$ of $X$.
In the corollaries we consider sufficient conditions to remove the dependence of $p^X$ from the estimate.

In Corollary \ref{cor:main} we assume that the transition probability measure $p^X$ has the bounded density function (see the assumption of Corollary \ref{cor:main}).
In this case, the existence of the fundamental solution to (\ref{PDE}) is obtained, and estimates follow from Theorem \ref{thm:main}.
As the consequence, under the additional assumption we obtain the $(1-\varepsilon)$-H\"older continuity of the fundamental solution in the spatial component.
When $a(t,\cdot)$ is locally Dini continuous uniformly in $t$, we have the $x(-\log x)$-order continuity of the fundamental solution.
Moreover, when $a(t,\cdot)$ is Dini continuous uniformly in $t$, we have the Lipschitz continuity.

When the coefficients $a$ and $b$ do not depend on $t$, there exists the density function of $p^X$.
Moreover, there are known estimates about the integrability of the density function (see Chapter 9 in \cite{StVa}).
Applying these estimates to Theorem \ref{thm:main}, we obtain the $(1-\varepsilon)$-H\"older continuity of the solution in a spatial component in the case that $a(t,\cdot)$ is continuous uniformly in $t$,
we have the $x(-\log x)$-order continuity in the case that $a(t,\cdot)$ is locally Dini continuous uniformly in $t$,
and we have the Lipschitz continuity in the case that $a(t,\cdot)$ is Dini continuous uniformly in $t$.

In the proof of the theorem, we express $u(t,x)$ by the Markov process $X$ and use the coupling method introduced by \cite{LiRo}.
The coupling method enables us to dominate the oscillation of $u(t,\cdot)$ by the oscillation of $X$, and as the consequence we have  a estimate of the oscillation of $u(t,\cdot)$ without the H\"older continuity of the coefficients $a$, $b$ and $c$.
In the estimate, the expectation of the coupling time appears. According to the upper bound of the expectation, we have the level of the continuity of $u(t,\cdot)$.

The organization of the present paper is as follows.
In Section \ref{sec:main}, we prepare notations and state the main theorem.
We also state the corollaries and prove them by applying the main theorem.
In Section \ref{sec:representation}, we consider the case that $a$ and $b$ are smooth, prove the H\"older continuity and show the dependence of the constant in the estimate.
This section is the main part of the present paper.
We use the coupling method to estimate the oscillation of the solution from the oscillation of the associated stochastic processes.
In Section \ref{sec:general}, we consider the case that $a$ and $b$ are not smooth, and prove the main result by using the result in Section \ref{sec:representation}.
The proof is done only by approximating $a$ and $b$ by smooth functions.
When we take the limit, we apply Stroock and Varadhan's result.

Now we give some notations.
Let $s \wedge t := \min \{ s,t\}$ and $s \vee t := \max \{ s,t\}$ for $s,t \in {\mathbb R}$.
For $p\in [1,\infty]$ denote the H\"older conjugate of $p$ by $p^*$.
Random variables in the present paper are considered on a probability space $(\Omega ,{\mathscr F}, P)$, we denote the expectation of random variables by $E[\, \cdot \, ]$ and the expectation on the event $A\in {\mathscr F}$ (i.e. $\int _A\, \cdot\, dP$) by $E[\, \cdot \, ; A]$.

\section{Main result}\label{sec:main}

Consider the following parabolic partial differential equation:
\begin{equation}\label{PDEX}
\left\{ \begin{array}{rl}
\displaystyle \frac{\partial}{\partial t} u^X(t,x) & \displaystyle = \frac 12 \sum _{i,j=1}^d a_{ij}(t,x)\frac{\partial ^2}{\partial x_i \partial x_j} u^X(t,x) +  \sum _{i=1}^d b_i(t,x) \frac{\partial }{\partial x_i} u^X(t,x) \\[3mm]
\displaystyle u^X(0,x)& =f(x). 
\end{array}\right.
\end{equation}
This equation is obtained by letting $c=0$ in (\ref{PDE}).
Define $d\times d$-matrix-valued function $\sigma (t,x)$ by the square root of $a(t,x)$.
Then, $a(t,x)= \sigma (t,x) \sigma (t,x)^T$ and
\begin{equation}\label{eq:sigma}
\sup _{t\in [0,\infty )}\sup _{i,j} |\sigma _{ij}(t,x)-\sigma _{ij}(t,y)| \leq C\rho _R (|x-y|) ,\quad x,y\in B(0;R),
\end{equation}
where $C$ is a constant depending on $\Lambda$.
Note that (\ref{ass:a1}) implies
\begin{equation}\label{eq:sigma2}
\Lambda ^{-1/2} I \leq \sigma (t,x) \leq \Lambda ^{1/2} I.
\end{equation}
Let $T>0$.
Consider the stochastic differential equation:
\begin{equation}\label{SDEX}
\left\{ \begin{array}{rl}
\displaystyle dX_t^x & \displaystyle = \sigma (T-t,X_t^x)dB_t + b(T-t,X_t^x)dt, \quad t\in [0,T]\\
\displaystyle X_0^x& \displaystyle =x .
\end{array}\right.
\end{equation}
From (\ref{eq:sigma}), (\ref{eq:sigma2}) and the boundedness of $b$, we have the existence and the uniqueness of the solution $X^x$ to (\ref{SDEX}) (see \cite{StVa}).
Denote the transition probability measure of $X$ by $p^X(s,x;t,dy)$.
We remark that $X$ and $p^X$ are depending on $T$.
The parabolic partial differential equation (\ref{PDEX}) and the stochastic differential equation (\ref{SDEX}) are associated with each other, and it holds that
\[
u^X(T,x) = E[f(X_T^x)] = \int _{{\mathbb R}^d} f(y) p^X(0,x;T,dy), \quad f\in C_b({\mathbb R}^d)
\]
(see (3.1) below for the detail).
Hence, considering (\ref{PDEX}) is equivalent to considering (\ref{SDEX}).

Before stating the main theorem, we prepare the following.

\begin{defi}\label{def:conti}
Let $f$ be a function on ${\mathbb R}^d$.
\begin{enumerate}
\item If for $R>0$ there exists a continuous and nondecreasing function $\rho _R$ on $[0,\infty )$ such that $\rho _R (0) =0$, $\int _0^1 r^{-1}\rho _R(r) dr <\infty$ and
\[
\left| f(x) - f(y)\right| \leq \rho _R (|x-y|) ,\quad x,y\in B(0;R),
\]
then $f$ is called locally Dini continuous.

\item If there exists a continuous and nondecreasing function $\rho $ on $[0,\infty )$ such that $\rho (0) =0$, $\int _0^1 r^{-1}\rho (r) dr <\infty$ and
\[
\left| f(x) - f(y)\right| \leq \rho (|x-y|) ,\quad x,y\in {\mathbb R}^d,
\]
then $f$ is called Dini continuous.
\end{enumerate}
\end{defi}

We remark that a Dini continuous function is locally Dini continuous and uniformly continuous.
It is easy to see that a H\"older continuous function and a locally H\"older continuous function are Dini continuous and locally Dini continuous, respectively.
It is also easy to see that; for $\alpha \in (1,\infty)$ a function $f$ on ${\mathbb R}^d$ which satisfies
\[
\left| f(x) - f(y)\right| \leq C\min\{ 1, (-\log |x-y|)^\alpha\} ,\quad x,y\in {\mathbb R}^d
\]
with a positive constant $C$, is Dini continuous.
Hence, the class of the Dini continuous functions is larger than the class of the H\"older continuous functions.

Let $f$ be a function on $[0,\infty )\times {\mathbb R}^d$.
If for $R>0$ there exists a continuous and nondecreasing function $\rho _R$ on $[0,\infty )$ such that $\rho _R (0) =0$ and
\[
\left| f(t,x) - f(t,y)\right| \leq \rho _R (|x-y|) ,\quad x,y\in B(0;R),
\]
we call $f(t,\cdot )$ is continuous uniformly in $t$.
If a function $f$ on $[0,\infty )\times {\mathbb R}^d$ such that $f(t,\cdot)$ is locally Dini continuous and the function $\rho _R$ appeared in Definition \ref{def:conti} can be chosen independently of $t$, then we call $f(t,\cdot )$ is locally Dini continuous uniformly in $t$.
Similarly, we define a function Dini continuous uniformly in $t$.

For a matrix-valued function $f(t,x)=(f_{ij}(t,x))$ on $[0,\infty )\times {\mathbb R}^d$, we say that $f$ is continuous uniformly in $t$ (resp. locally Dini continuous uniformly in $t$ and Dini continuous function uniformly in $t$) if all components of $f$ are continuous uniformly in $t$ (resp. locally Dini continuous uniformly in $t$ and Dini continuous function uniformly in $t$).

In the main theorem, we will assume that the coefficient $a(t,x)$ of the second order derivative term is continuous in $x$ uniformly in $t$.
The assumption is equivalent to the existence of $\rho _R$ such that (\ref{ass:a2}) holds.
As a rule of the present paper, $\rho _R$ will be regarded as the functions satisfying (\ref{ass:a2}), once $a(t,\cdot )$ is assumed to be continuous uniformly in $t$.
When $a(t,\cdot )$ is assumed to be locally Dini continuous uniformly in $t$, $\rho _R$ will be regarded as the functions satisfying (\ref{ass:a2}) and $\int _0^1 r^{-1}\rho _R(r) dr <\infty$ for $R>0$.
Furthermore, when $a(t,\cdot)$ is assumed to be Dini continuous uniformly in $t$, $\rho$ will be regarded as the function satisfying (\ref{ass:a2}) with replacement $\rho _R$ by $\rho$ and $B(0;R)$ by ${\mathbb R}^d$, and satisfying $\int _0^1 r^{-1}\rho(r) dr <\infty$.

Now we state the main result of the present paper.

\begin{thm}\label{thm:main}
Assume (\ref{ass:a1}) and that $a(t,\cdot )$ is continuous uniformly in $t$, and let $u$ be the solution of (\ref{PDE}).

\begin{enumerate}
\item \label{thm:main1}
For any $p\in [1,\infty)$, $R>0$ and sufficiently small $\varepsilon >0$, there exists a constant $C$ depending on $d$, $\Lambda$, $\varepsilon$, $R$, $\rho _R$, $\| b\| _\infty$ and $\| c\| _\infty$ such that
\begin{align*}
|u(T,x)-u(T,z)| &\leq Cs^{-(1+1/p)}Te^{CT} |x-z|^{(1-\varepsilon )/p} \\
&\quad \times \max_{\eta = x,z} \left\| \int _{{\mathbb R}^d} |f(y)| p^{X}\left( s,\cdot ;T, dy\right) \right\| _{L^{p^*}({\mathbb R}^d, p^{X}\left( 0,\eta ;s, \cdot \right))}
\end{align*}
for $f\in C_b({\mathbb R}^d)$, $s\in (0,T)$, and $x,z \in B(0;R/2)$.

\item \label{thm:main2} Additionally assume that $a(t,\cdot )$ is locally Dini continuous uniformly in $t$.
Then, for any $p\in [1,\infty)$ and $R>0$, there exists a constant $C$ depending on $d$, $\Lambda$, $R$, $\rho _R$, $\| b\| _\infty$ and $\| c\| _\infty$ such that
\begin{align*}
|u(T,x)-u(T,z)| &\leq Cs^{-(1+1/p)}Te^{CT} (|x-z| \max\{ 1, -\log |x-z|\} )^{1/p}\\
&\quad \times \max_{\eta = x,z} \left\| \int _{{\mathbb R}^d} |f(y)| p^{X}\left( s,\cdot ;T, dy\right) \right\| _{L^{p^*}({\mathbb R}^d, p^{X}\left( 0,\eta ;s, \cdot \right))}
\end{align*}
for $f\in C_b({\mathbb R}^d)$, $s\in (0,T)$, and $x,z \in B(0;R/2)$.

\item \label{thm:main3}
Additionally assume that $a(t,\cdot )$ is Dini continuous uniformly in $t$.
Then, for any $p\in [1,\infty)$, there exists a constant $C$ depending on $d$, $\Lambda$, $\rho$, $\| b\| _\infty$ and $\| c\| _\infty$ such that
\begin{align*}
|u(T,x)-u(T,z)| &\leq Cs^{-(1+1/p)}Te^{CT} |x-z|^{1/p} \\
&\quad \times \max_{\eta = x,z} \left\| \int _{{\mathbb R}^d} |f(y)| p^{X}\left( s,\cdot ;T, dy\right) \right\| _{L^{p^*}({\mathbb R}^d, p^{X}\left( 0,\eta ;s, \cdot \right))}
\end{align*}
for $f\in C_b({\mathbb R}^d)$, $s\in (0,T)$, and $x,z \in {\mathbb R}^d$.
\end{enumerate}
\end{thm}

\begin{rem}
Letting $s= T/2$ and $p=1$ in Theorem \ref{thm:main} {\ref{thm:main3}}, we obtain
\[
|u(T,x)-u(T,z)| \leq CT^{-1} e^{CT} |x-z| \| f\| _\infty
\]
for $f\in C_b({\mathbb R}^d)$, $T\in (0,\infty )$ and $x,z \in {\mathbb R}^d$.
In the case that $a$, $b$ and $c$ do not depend on the time component, $a$ is Dini continuous and $c\leq 0$, Priola and Wang has obtained Lipschitz continuity of $u(t,\cdot )$ with Lipschitz constant $\frac{C}{\sqrt{t\wedge 1}}$ by a similar idea without the boundedness of the coefficients (see \cite{PrWa}).
In this paper we assume the boundedness of the all coefficients, because we starts with the assumptions in {\ref{thm:main1}} and obtain {\ref{thm:main3}} as a specific case.
In \cite{PoPr} the result in \cite{PrWa} is shown by a purely analytic way. Moreover, they concerns the case of unbounded continuous $f$ with a growth condition, and obtain the Lipschitz continuity of $u(t,\cdot )$.
We remark that in \cite{PoPr} the result is stated for viscosity solutions.
\end{rem}

Adding a little more assumption on $p^X$, we can remove $p^X$ from the upper bounds in the estimates in Theorem \ref{thm:main} and have the modulus of the continuity of the fundamental solution to $(\ref{PDE})$.
To state the definition of the fundamental solution, denote
\[
L_t f(x) := \frac 12 \sum _{i,j=1}^d a_{ij}(t,x)\frac{\partial ^2}{\partial x_i \partial x_j} f(x)+  \sum _{i=1}^d b_i(t,x) \frac{\partial }{\partial x_i}f(x) + c(t,x) f(x) ,\quad f\in C^2_b ({\mathbb R}^d).
\]
A measurable function $p(s,x;t,y)$ defined for $s,t\in [0,\infty)$ such that $s<t$, and $x,y \in {\mathbb R}^d$ is called a fundamental solution to (\ref{PDE}), if $p(s,x;t,y)$ satisfies
\begin{align*}
& \frac{\partial}{\partial t} \int _{{\mathbb R}^d} f(y)p(s,\cdot  ;t,y) dy = L_t \left( \int _{{\mathbb R}^d} f(y)p(s,\cdot ;t,y) dy \right) \\
& \lim _{r\downarrow s}  \int _{{\mathbb R}^d} f(y)p(s,\cdot  ;r,y) dy = f
\end{align*}
for $f\in C_b^2({\mathbb R}^d)$ with a compact support.
When $c=0$ and $p^X(s,x;T,dy)$ is absolutely continuous with respect to the Lebesgue measure, $p(s,x;T,y)$ is obtained by
\begin{equation}\label{eq:fund-trans}
p(s,x;T,y) = \frac{p^X(s,x;T,dy)}{dy}
\end{equation}
where the right-hand side is the Radon-Nikodym derivative of $p^X(s,x;T,dy)$ with respect to the Lebesgue measure.
The right-hand side of (\ref{eq:fund-trans}) is called the transition probability density function.
In view of this fact, we define the fundamental solution $p(s,x;t,y)$ as above.
However, the order of the parameters $s,x,t,y$ in the present paper is a little different from the standard.
For example, it is different from the notation in \cite{Fr}.

Now we state a corollary of Theorem \ref{thm:main}.
In the following corollary, we assume the existence of the density function of $p^X(T/2,x;T,\cdot)$ and the bounds of the density function.
In this case, we have the fundamental solution $p(T/2,x;T,y)$ to (\ref{PDE}), and moreover we obtain the modulus of the continuity of $p(s,\cdot ;t,y)$, as follows.

\begin{cor}\label{cor:main}
Assume (\ref{ass:a1}) and that $a(t,\cdot )$ is continuous uniformly in $t$.
Moreover, we assume that for each $(T,x)\in (0,\infty )\times {\mathbb R}^d$, $p^{X}\left( T/2,x;T, \cdot \right)$ is absolutely continuous with respect to the Lebesgue measure and there exists a continuous function $\nu (t)$ on $(0,\infty)$ such that
\[
\sup _{x\in {\mathbb R}^d} \frac{p^{X}\left( T/2,x;T, dy\right)}{dy} \leq \nu (T), \quad T\in (0,\infty).
\]
Then, there exists a measurable function $p(0,x;t,y)$ on $t\in (0,\infty)$ and $x,y\in {\mathbb R}^d$ which satisfies the definition of the fundamental solution to (\ref{PDE}) under the restriction that $s=0$, and the followings hold.

\begin{enumerate}
\item \label{cor:main1-1} For $R>0$ and sufficiently small $\varepsilon >0$, there exists a constant $C$ depending on $d$, $\Lambda$, $\varepsilon$, $R$, $\rho _R$, $\| b\| _\infty$ and $\| c\| _\infty$ such that
\begin{align*}
|p(0,x;t,y)-p(0,z;t,y)| &\leq C t^{-1} \nu (t) e^{Ct} |x-z|^{1-\varepsilon }
\end{align*}
for $t\in (0,\infty )$, almost every $y\in {\mathbb R}^d$ with respect to the Lebesgue measure, and $x,z \in B(0;R/2)$.

\item \label{cor:main1-2}
Additionally assume that $a(t,\cdot )$ is locally Dini continuous uniformly in $t$.
Then, for $R>0$, there exists a constant $C$ depending on $d$, $\Lambda$, $R$, $\rho _R$, $\| b\| _\infty$ and $\| c\| _\infty$ such that
\begin{align*}
|p(0,x;t,y)-p(0,z;t,y)| &\leq C t^{-1} \nu (t) e^{Ct} |x-z| \max\{ 1, -\log |x-z|\} 
\end{align*}
for $t\in (0,\infty )$, almost every $y\in {\mathbb R}^d$ with respect to the Lebesgue measure,  and $x,z \in B(0;R/2)$.

\item \label{cor:main1-3}
Additionally assume that $a(t,\cdot )$ is Dini continuous uniformly in $t$.
Then, there exists a constant $C$ depending on $d$, $\Lambda$, $\rho$, $\| b\| _\infty$ and $\| c\| _\infty$ such that
\begin{align*}
|p(0,x;t,y)-p(0,z;t,y)| &\leq C t^{-1} \nu (t) e^{Ct} |x-z|
\end{align*}
for $t\in (0,\infty )$, almost every $y\in {\mathbb R}^d$ with respect to the Lebesgue measure,  and $x,z \in {\mathbb R}^d$.
\end{enumerate}
\end{cor}

\begin{proof}
Applying \ref{thm:main1} of Theorem \ref{thm:main} with $p=1$ and $s=T/2$, we have that; for $R>0$ and sufficiently small $\varepsilon >0$, there exists a constant $C$ depending on $d$, $\Lambda$, $\varepsilon$, $R$, $\rho _R$, $\| b\| _\infty$ and $\| c\| _\infty$ such that
\[
|u(T,x)-u(T,z)| \leq C T^{-1} \nu (T) e^{CT} |x-z|^{1-\varepsilon } \| f\| _{L^1({\mathbb R}^d)}
\]
for $f\in C_b({\mathbb R}^d)$ such that ${\rm supp} f \subset B(0;R)$, $T\in (0,\infty )$ and $x,z \in B(0;R/2)$.
This inequality implies the existence of the fundamental solution $p(0,x;t,y)$ and \ref{cor:main1-1}.

The proofs of \ref{cor:main1-2} and \ref{cor:main1-3} are obtained by applying \ref{thm:main2} and \ref{thm:main3} of Theorem \ref{thm:main} instead of \ref{thm:main1} of Theorem \ref{thm:main}, respectively.
\end{proof}

\begin{rem}\label{rem:counterexample}\rm
\begin{enumerate}
\item When the upper Gaussian estimate holds for $p^X$, then $\nu (T)$ can be chosen by $CT^{-d/2}$ with a constant $C$.

\item The order of $|x-z|$ and the order of $t$ for small $t$ in \ref{cor:main1-3} of Corollary \ref{cor:main} is optimal under the condition.
Consider the one-dimensional parabolic equation:
\[
\frac{\partial}{\partial t} u(t,x) = \frac 12 \frac{\partial ^2}{\partial x^2} u(t,x) - \theta {\rm sgn}(x) \frac{\partial }{\partial x} u(t,x)
\]
where $\theta >0$ and the function ${\rm sgn}$ is defined by ${\rm sgn} x := x/|x|$ for $x \neq 0$ and ${\rm sgn}0 :=0$.
The fundamental solution $p(0,x;t,y)$ of this equation is obtained explicitly as
\begin{align*}
&p(0,x;t,y)\\
& = \left\{ \begin{array}{l}
\displaystyle \frac{1}{\sqrt{2\pi t}}\left[ \exp\left\{ -\frac{(x-y-\theta t)^2}{2t}\right\} + \theta e^{-2\theta y} \int _{x+y}^\infty \exp\left\{ -\frac{(\xi -\theta t)^2}{2t}\right\} d\xi \right] \\[5mm]
\displaystyle \hspace{10cm} x\geq 0,\ y\geq 0,\\[5mm]
\displaystyle \frac{1}{\sqrt{2\pi t}}\left[ \exp\left\{ 2\theta x -\frac{(x-y+\theta t)^2}{2t}\right\} + \theta e^{2\theta y} \int _{x-y}^\infty \exp\left\{ -\frac{(\xi -\theta t)^2}{2t}\right\} d\xi \right] \\[5mm]
\displaystyle \hspace{10cm} x\geq 0,\ y< 0,\\[5mm]
\displaystyle \frac{1}{\sqrt{2\pi t}}\left[ \exp\left\{ -2\theta x -\frac{(x-y-\theta t)^2}{2t}\right\} + \theta e^{-2\theta y} \int _{-x+y}^\infty \exp\left\{ -\frac{(\xi -\theta t)^2}{2t}\right\} d\xi \right] \\[5mm]
\displaystyle \hspace{10cm} x< 0,\ y\geq 0,\\[5mm]
\displaystyle \frac{1}{\sqrt{2\pi t}}\left[ \exp\left\{ -\frac{(x-y+\theta t)^2}{2t}\right\} + \theta e^{2\theta y} \int _{-x-y}^\infty \exp\left\{ -\frac{(\xi -\theta t)^2}{2t}\right\} d\xi \right] \\[5mm]
\displaystyle \hspace{10cm} x< 0,\ y< 0.
\end{array}\right.
\end{align*}
 (see Remark 5.2 of Chapter 6 in \cite{KaSh}).
In this case, $p(0,x;t,y)$ is Lipschitz continuous. However it does not belong to $C^1({\mathbb R})$.
This implies that the obtained order of $|x-z|$ in \ref{cor:main1-3} of Corollary \ref{cor:main} is optimal.
As in the previous remark, $\nu (T)$ can be chosen by $CT^{-1/2}$ with a constant $C$.
Hence, the obtained order of $t$ in \ref{cor:main1-3} of Corollary \ref{cor:main} is $t^{-3/2}$ for small $t$.
This order coincides with the order obtained by the explicit calculation of the example.
\end{enumerate}
\end{rem}

Next, we consider another case that $p^X$ is to be removed from the upper bounds in the estimates in Theorem \ref{thm:main}.
Stroock and Varadhan deeply studied the properties of the transition density functions of the solutions to stochastic differential equations with low regular coefficients in \cite{StVa}.
As one of their results, it is known that when the coefficients do not depend on time, then the transition probability density function exists for all time and the transition probability density function is in $L^p$ for all $p \in [1,\infty)$ (see Corollary 9.2.7 in \cite{StVa}).
Applying this result to Theorem \ref{thm:main}, we can remove $p^X$ from the upper estimates and obtain the clearer modulus of continuity of the solutions to (\ref{PDE}), as follows.

\begin{cor}\label{cor:main2}
Assume that $a$ and $b$ do not depend on $t$.
Assume (\ref{ass:a1}) and that $a$ is continuous, and let $u$ be the solution of (\ref{PDE}).

\begin{enumerate}
\item \label{cor:main2-1}
For any $p\in (1,\infty]$, $R>0$, $\tilde T>0$, and sufficiently small $\varepsilon >0$, there exist constants $\alpha$ depending on $d$ and $p$, and $C$ depending on $\tilde T$, $d$, $p$, $\Lambda$, $\varepsilon$, $R$, $\rho _R$, $\| b\| _\infty$ and $\| c\| _\infty$ such that
\begin{align*}
|u(t,x)-u(t,z)| &\leq Ct^{-\alpha} |x-z|^{1-\varepsilon} \|f \| _{L^{p}({\mathbb R}^d)}
\end{align*}
for $f\in L^p({\mathbb R}^d)$ such that ${\rm supp} f \subset B(0;R/2)$, $t\in (0,\tilde T]$ and $x,z \in B(0;R/2)$.

\item \label{cor:main2-2}
Additionally assume that $a$ is locally Dini continuous.
Then, for any $p\in (1,\infty]$, $R>0$ and $\tilde T>0$, there exist constants $\alpha$ depending on $d$ and $p$, and $C$ depending on $\tilde T$, $d$, $p$, $\Lambda$, $R$, $\rho _R$, $\| b\| _\infty$ and $\| c\| _\infty$ such that
\begin{align*}
|u(t,x)-u(t,z)| &\leq Ct^{-\alpha} |x-z| \max\{ 1, -\log |x-z|\} \|f \| _{L^{p}({\mathbb R}^d)}
\end{align*}
for $f\in L^p({\mathbb R}^d)$ such that ${\rm supp} f \subset B(0;R/2)$, $t\in (0,\tilde T]$ and $x,z \in B(0;R/2)$.

\item \label{cor:main2-3}
Additionally assume that $a$ is Dini continuous.
Then, for any $p\in (1,\infty]$ and $\tilde T>0$, there exist constants $\alpha$ depending on $d$ and $p$, and $C$ depending on $\tilde T$, $d$, $p$, $\Lambda$, $\rho$, $\| b\| _\infty$ and $\| c\| _\infty$ such that
\begin{align*}
|u(t,x)-u(t,z)| &\leq Ct^{-\alpha} |x-z| \|f \| _{L^{p}({\mathbb R}^d)}
\end{align*}
for $f\in L^p({\mathbb R}^d)$, $t\in (0,\tilde T]$ and $x,z \in {\mathbb R}^d$.
\end{enumerate}
\end{cor}

\begin{proof}
Note that $p^X(s,x;T, \cdot)$ is absolutely continuous with respect to the Lebesgue measure in the case that $a,b$ do not depend on $t$ (see Lemma 9.2.2 in \cite{StVa}).
We denote $\frac{p^X(s,x;T,dy)}{dy}$ by $p^X(s,x;T,y)$.
Applying \ref{thm:main1} of Theorem \ref{thm:main} with $s=T/2$ and $p=1$, we have
\begin{equation}\label{eq:cormain2-2-0}
|u(T,x)-u(T,z)| \leq CT^{-1}e^{CT} |x-z|^{1-\varepsilon} \left\| \int _{{\mathbb R}^d} |f(y)| p^{X}\left( T/2,\cdot ;T, y\right) dy\right\| _{L^{\infty}({\mathbb R}^d)}
\end{equation}
for $f\in C_b({\mathbb R}^d)$, $T\in (0,\infty )$ and $x,z \in B(0;R/2)$, where $C$ is a constant depending on $d$, $\Lambda$, $\varepsilon$, $R$, $\rho _R$, $\| b\| _\infty$ and $\| c\| _\infty$.
On the other hand, Corollary 9.2.7 in \cite{StVa} implies that there exist constants $\beta$ depending on $d$ and $p$, and $C$ depending on $\tilde T$, $d$, $p$, $\Lambda$, $R$, $\rho _R$, $\| b\| _\infty$ and $\| c\| _\infty$ such that
\begin{equation}\label{eq:cormain2-2-1}
\sup _{\xi \in {\mathbb R}^d} \left( \int _{B(0;R/2)} |p^X(T/2,\xi ;T,y)|^{p^*} dy \right) ^{1/p^*} \leq C(T\wedge 1)^{-\beta}, \quad T\in (0,\tilde T] .
\end{equation}
Hence, by the H\"older's inequality
\begin{align*}
\int _{B(0;R)} |f(y)| p^{X}\left( T/2,\xi ;T, y\right) dy
&\leq \| f\| _{L^p({\mathbb R}^d)} \left( \int _{B(0;R/2)} |p^X(T/2,x;T,y)|^{p^*} dy \right) ^{1/p^*} \\
&\leq C(T\wedge 1)^{-\beta} \| f\| _{L^p({\mathbb R}^d)}
\end{align*}
for $f\in C_b({\mathbb R}^d)$ such that ${\rm supp} f \subset B(0;R/2)$ and $T\in (0, \tilde T]$.
Therefore, by (\ref{eq:cormain2-2-0}) we obtain the assertion for $f\in C_b({\mathbb R}^d)$ such that ${\rm supp} f \subset B(0;R/2)$.

Applying \ref{thm:main2} of Theorem \ref{thm:main} with $s=T/2$, we obtain \ref{cor:main2-2}.
The proof of \ref{cor:main2-3} is obtained by Theorem 9.2.6 in \cite{StVa} instead of Corollary 9.2.7 in \cite{StVa}.
\end{proof}

\begin{rem}
In Corollary \ref{cor:main2} $c$ can depend on the time component.
If $c$ is also independent of the time component, then the theory of analytic semigroup is applicable and a better result is obtained (see Section 3.1.1 in \cite{Lu}).
\end{rem}

\section{Probabilistic representation of the fundamental solution}\label{sec:representation}

In this section, we assume that $a_{ij}(t,x) \in C_b^\infty ([0,\infty)\times {\mathbb R}^d)$ and $b_i(t,x) \in C_b^\infty([0,\infty)\times {\mathbb R}^d)$, and will obtain an a priori estimate.
Let $T>0$.
Define $d\times d$-matrix-valued function $\sigma (t,x)$ by the square root of $a(t,x)$ and consider the stochastic differential equation (\ref{SDEX}).
Since $a_{ij}(t,x) \in C_b^\infty ([0,\infty)\times {\mathbb R}^d)$, (\ref{ass:a1}) implies that $\sigma_{ij}(t,x) \in C_b^\infty ([0,\infty)\times {\mathbb R}^d)$.
Note that $a(t,x)= \sigma (t,x) \sigma (t,x)^T$, (\ref{eq:sigma}) and (\ref{eq:sigma2}) hold.
Lipschitz continuity of $\sigma$ and $b$ implies that the existence of the solution and the pathwise uniqueness hold for (\ref{SDEX}).
Let $({\mathscr F}_t)$ be the $\sigma$-field generated by $(B_s; s\in [0,t])$.
Then, the pathwise uniqueness implies that the solution $X^x_t$ is ${\mathscr F}_t$-measurable for $t\in [0,\infty)$.
All stopping times which appear in this paper are regarded with respect to $({\mathscr F}_t)$.
We remark that the generator $(L_t^X)$ of $(X_t^x)$ is given by
\[
L_t^X = \frac 12 \sum _{i,j=1}^d a_{ij}(t,x)\frac{\partial ^2}{\partial x_i \partial x_j} + \sum _{i=1}^d b_i(t,x)\frac{\partial}{\partial x_i}.
\]
The smoothness of $\sigma$ and $b$ implies the existence of the fundamental solution $p^X(s,x;t,y)$ associated with $L_t^X$ and the smoothness on $(0,\infty ) \times {\mathbb R}^d \times (0,\infty ) \times {\mathbb R}^d$ (see e.g  \cite{KuSt} for probabilistic approach or \cite{LM} for analytic approach).

Consider the following backward parabolic equation on $[0,T]\times {\mathbb R}^d$
\begin{equation}\label{PDEv}
\left\{ \begin{array}{rl}
\displaystyle -\frac{\partial}{\partial t} v(t,x) \!\!\!\! & \displaystyle = \frac 12 \sum _{i,j=1}^d a_{ij}(T-t,x)\frac{\partial ^2}{\partial x_i \partial x_j} v(t,x) +  \sum _{i=1}^d b_i(T-t,x) \frac{\partial }{\partial x_i} v(t,x) \\
\!\!\!\! &\displaystyle \hspace{7.5cm} + c(T-t,x) v(t,x) \\
\displaystyle v(T,x) \!\!\!\!& =f(x).\displaystyle 
\end{array}\right.
\end{equation}
Then, it is easy to see the one-to-one correspondence between the solution $u$ to (\ref{PDE}) and the solution $v$ to (\ref{PDEv}) by $u(t,x) = v(T-t,x)$ for $(t,x) \in [0,T] \times {\mathbb R}^d$.
By the Feynman-Kac formula (see e.g. Proposition 3.10 of Chapter VIII in \cite{ReYo}), we have the following representation of $u(t,x)$ by $(X_t^x)$.
\begin{equation}\label{eq:uX}
u(T,x)= v(0,x) =E\left[ f(X_T^x)\exp \left( \int _0^T c(T-s,X_s^x)ds\right) \right] .
\end{equation}
Hence, to see the regularity of $u(T,\cdot )$, it is sufficient to see the regularity of the function $x \mapsto E\left[ f(X_T^x)\exp \left( \int _0^T c(T-s,X_s^x)ds\right) \right]$.
By the way, by the definition of the fundamental solution and (\ref{eq:uX}), we obtain the probabilistic representation of the fundamental solution:
\begin{equation}\label{eq:fundamental}
p(0,x;T,y) = p^X(0,x;T,y) E^{X^x_T =y}\left[ \exp \left( \int _0^T c(T-s,X_s^x)ds\right) \right]
\end{equation}
where $P^{X_T^x = y}$ is the conditional probability measure of $P$ on $X_T^x = y$ and $E^{X^x_T=y}[\, \cdot \, ]$ is the expectation with respect to $P^{X_T^x = y}$.
For the calculation of the conditional probabilities we will use the following equality (see (2.10) in \cite{Sei010-1}).
\begin{equation}\label{eq:conditional}
P^{X_t^x = y} \left( A \right) = \frac{1}{p^X(0,x;t,y)} \int _{{\mathbb R}^d} p^X(s,\xi ;t,y) P\left( A\cap\{ X_s^x \in d\xi \} \right)
\end{equation}
for $s,t\in (0,\infty )$ such that $s<t$, $A\in {\mathscr F}_s$ and $x,y \in {\mathbb R}^d$.

Now, to have estimates of the oscillation of $u(T,\cdot )$, we apply the coupling method introduced by \cite{LiRo}.
According to $(X^x, B)$ defined by (\ref{SDEX}), we consider the stochastic process $(Z_t^z)$ defined by
\begin{equation}\label{SDEz}
\left\{ \begin{array}{rl}
\displaystyle Z_t^z \!\!\!\! & \displaystyle = z + \int _0 ^{t\wedge \tau} \sigma (T-s,Z_s^z)d\tilde B_s + \int _{t\wedge \tau} ^t \sigma (T-s,Z_s^z)dB_s + \int _0^t b(T-s,Z^z_s) ds,\\
\!\!\!\! & \hspace{11cm} t\in [0,T],\\[3mm]
\displaystyle \tilde B_t \!\!\!\! & \displaystyle = \int _0^{t\wedge \tau} \!\!\! \left( I -\frac{2(\sigma (T-s,Z_s^z)^{-1} (X_s^x-Z_s^z))\otimes (\sigma (T-s,Z_s^z)^{-1} (X_s^x-Z_s^z))}{|\sigma (T-s,Z_s^z)^{-1} (X_s^x-Z_s^z)|^2}\right) dB_s
\end{array}\right.
\end{equation}
where $\tau$ is the stopping time defined by $\tau := \inf \{ t\geq 0; X_t^x=Z_t^z\} \wedge T$.
By the Lipschitz continuity of $\sigma$ and $b$, $(Z^z(t), \tilde B(t); t\in [0,\tau ))$ are determined almost surely and uniquely.
Let
\[
H_t :=I -\frac{2\left[ \sigma (T-t,Z_t^z)^{-1} (X_t^x-Z_t^z)\right] \otimes \left[ \sigma (T-t,Z_t^z)^{-1} (X_t^x-Z_t^z)\right] }{|\sigma (T-t,Z_t^z)^{-1} (X_t^x-Z_t^z)|^2} .
\]
Then, $H_t$ is an orthogonal matrix for all $t\in [0,\tau )$, and hence $\tilde B_t$ is a $d$-dimensional Brownian motion for $t\in [0,\tau )$.
Hence, $(Z^z(t), \tilde B(t); t\in [0,\tau ))$ are extended to $(Z^z(t), \tilde B(t); t\in [0,\tau ])$ almost surely and uniquely.
By the Lipschitz continuity of $\sigma$ and $b$ again, (\ref{SDEz}) is solved almost surely and uniquely for $t\in [\tau , T]$.
Thus, we obtain $(Z^z(t); t\in [0,T])$ almost surely and uniquely.
From this fact we have that $Z^z_t$ is ${\mathscr F}_t$-measurable for $t \in [0,T]$.
Besides, the argument above implies that; if $x=z$, $X^x$ and $Z^z$ has the same law.
Moreover, $X^x_t = Z^z_t$ for $t\in [\tau ,T]$ almost surely.

\begin{lem}\label{lem:estprob1}
For $R>0$ and sufficiently small $\varepsilon >0$, there exists a positive constants $C$ depending on $d$, $\Lambda$, $\varepsilon$, $R$, $\rho _R$ and $\| b\| _\infty$ such that
\begin{equation}\label{eq:estprob1-0}
E[t\wedge \tau] \leq C(1+t^2)|x-z|^{1-\varepsilon}
\end{equation}
for $t\in [0,T]$ and $x,z \in B(0;R/2)$.
\end{lem}

\begin{proof}
We remark that the proof is almost same as the proof of Lemma 4.2 in \cite{Sei010-1}.
Let $R>0$ and $x,z \in B(0;R/2)$.
Define
\[
\xi _t := X_t^x-Z_t^z, \quad \alpha _t := \sigma (T-t,X_t^x) - \sigma (T-t,Z_t^z)H_t, \quad \beta _t := b(T-t,X_t^x) -b(T-t,Z_t^z).
\]
Then, by It\^o's formula we have for $t\in [0,\tau)$
\begin{equation}\label{eq:xi}
d(|\xi _t|) = \left\langle \frac{\xi _t}{|\xi _t|}, \alpha _t dB_t \right\rangle + \frac 1{2|\xi _t|} \left( {\rm tr}( \alpha _t \alpha _t ^T) - \frac{|\alpha _t^T \xi _t|^2}{|\xi _t|^2} + \left\langle \xi _t, \beta _t \right\rangle \right) dt 
\end{equation}
where ${\rm tr}(A)$ is the trace of the matrix $A$.
By following the argument in the proof of Lemma 4.2 in \cite{Sei010-1},
we have a positive constant $\gamma _1$ depending on $d$ and $\Lambda$ such that
\begin{equation}\label{eq:coup01}
\left| {\rm tr}( \alpha _t \alpha _t ^T) - \frac{|\alpha _t^T \xi _t|^2}{|\xi _t|^2} + \left\langle \xi _t, \beta _t \right\rangle\right| \leq \gamma _1 \rho _R (|\xi _t|) ,\ t\in [0,\tau)\ \mbox{such that}\ X_t^x, Z_t^z \in B(0;R),
\end{equation}
and we also have a positive constant $\gamma _2$ depending on $d$ and $\Lambda$ such that
\begin{equation}\label{eq:coup02}
\frac{|\alpha _t ^T \xi _t|}{|\xi_t|} \geq \gamma_2 ^{-1} \ \mbox{for}\ t\in [0,\tau)\ \mbox{such that}\ |\sigma (T-t,X_t^x)-\sigma (T-t,Z_t^z)|\leq 2\Lambda ^{-1}.
\end{equation}
Note that if $\rho _R (|\xi _t|) \leq 2\Lambda ^{-1}$ and $X_t^x, Z_t^z\in B(0;R)$, then $|\sigma (T-t,X_t^x)-\sigma (T-t,Z_t^z)|\leq 2\Lambda ^{-1}$.
Let $\gamma := \gamma _1 \vee \gamma _2$.
For given $\varepsilon >0$, let
\[
\tilde \tau := \tau \wedge \inf \left\{ t\in [0,T];\ \rho _R (|\xi _t|) > \frac{\varepsilon}{2\gamma ^3} \wedge 2\Lambda ^{-1},\ X_t^x \not \in B(0;R) \ \mbox{or}\   Z_t^z \not \in B(0;R) \right\}
\]
Then, by following the argument in the proof of Lemma 4.2 in \cite{Sei010-1} again, we obtain
\begin{equation}\label{eq:est-nhit}
E\left[ t\wedge \tilde \tau \right] \leq C|x-z|^{1-\varepsilon} \quad \mbox{for}\ t\in [0,T]
\end{equation}
where $C$ is a constant depending on $d$, $\Lambda$, $\varepsilon$, $R$ and $\| b\| _\infty$.

Now we consider the estimate of the expectation of $\tau$ by using that of $\tilde \tau$.
We remark that, the following argument is almost same as the argument in the proof of Lemma 4.2 in \cite{Sei010-1}, however the estimate which we will obtain here is a little better than the estimate obtained in the proof of Lemma 4.2 in \cite{Sei010-1}.
To simplify the notation, let
\[
\delta _0 := \frac 13 \rho _R ^{-1} \left( \frac{\varepsilon}{2\gamma ^3} \wedge 2\Lambda ^{-1}\right) .
\]
Since
\begin{align*}
& |\xi _t| > 3\delta _0 \Longrightarrow |X_t^x -x| >\delta _0 ,\ |Z_t^z -z| >\delta _0,\ \mbox{or}\ |x-z|>\delta _0,\\
& X_t^x \not \in B(0;R) \ \mbox{or}\   Z_t^z \not \in B(0;R) \Longrightarrow |X_t^x -x| >\frac R2 \ \mbox{or}\ |Z_t^z -z| >\frac R2,
\end{align*}
we have for $x,z \in B(0;R/2)$ such that $|x-z| \leq \delta _0$,
\begin{equation}\label{eq:estprob1-2}\begin{array}{l}
\displaystyle P(\tau \geq t) \leq P(\tilde \tau \geq t) +P\left( \sup _{s\in [0,t]} |X_s^x -x| >\delta _0 \wedge \frac R2 \right) \\
\displaystyle \hspace{6cm} +P\left( \sup _{s\in [0,t]} |Z_s^z -z| >\delta _0 \wedge \frac R2 \right) .
\end{array}\end{equation}
Let $\eta = x$ or $z$, and let $\delta _1:= \delta _0 \wedge (R/2)$.
Proposition \ref{prop:appendix1} implies that
\begin{align*}
& P\left( \sup _{s\in [0,t]}|X_s^\eta -\eta | > \delta _1 \right) \\
& \leq P\left( \sup _{s\in [0,t]} \int _0^s \sigma (T-u,X^\eta _u) dB_u > \frac{\delta _1}{2} \right) + P\left( \sup _{s\in [0,t]}\left( -\int _0^s \sigma (T-u,X^\eta _u) dB_u \right) > \frac{\delta _1}{2} \right)\\
& \quad + P\left( \sup _{s\in [0,t]}\left| \int _0^s b(T-u,X^\eta _u) du \right| > \frac{\delta _1}{2} \right) \\
& \leq C_1\sqrt{t}\exp \left( -C_2t^{-1}\right) + P\left( \| b\| _\infty t > \frac{\delta _1}{2} \right)
\end{align*}
where $C_1$, $C_2$ are constants depending on $d$, $\Lambda$, $\varepsilon$, $R$ and $\rho _R$.
Hence, there exist positive constants $c$ and $C$ depending on $d$, $\Lambda$, $\varepsilon$, $R$, $\rho _R$ and $\| b\| _\infty$ such that
\begin{equation}\label{eq:estprob1-1}
P\left( \sup _{s\in [0,t]}|X_s^\eta -\eta | >\delta _1 \right) \leq C|x-z|
\end{equation}
for $\eta = x,z$ and $t\in [0, c(-\log |x-z| )^{-1} \wedge 1)]$.
By (\ref{eq:est-nhit}), (\ref{eq:estprob1-2}) and (\ref{eq:estprob1-1}) we have positive constants $c$ and $C$ depending on $d$, $\Lambda$, $\varepsilon$, $R$, $\rho _R$ and $\| b\| _\infty$ such that; for $x,z \in B(0;R/2)$ satisfying $|x-z| \leq c$, and $t\in [0, c(-\log |x-z| )^{-1} ]$
\begin{align*}
&E[t\wedge \tau] \\
&\leq \int _0^t P(\tau \geq s) ds \\
&\leq \int _0^t P( \tilde \tau \geq s) ds + t \left[ P\left( \sup _{s\in [0,t]}|X_s^x -x| >\delta _1 \right) + P\left( \sup _{s\in [0,t]}|X_s^z -z| >\delta _1 \right) \right] \\
&\leq C(1+t)|x-z|^{1-\varepsilon} .
\end{align*}
Therefore, we obtain
\begin{equation}\label{eq:estprob1-9}
E[t\wedge \tau] \leq C(1+t)|x-z|^{1-\varepsilon}
\end{equation}
for $x,z \in B(0;R/2)$ such that $|x-z| \leq c_0$, and $t\in [0, c_0(-\log |x-z| )^{-1} ]$ where $c_0$ and $C$ are positive constants depending on $d$, $\Lambda$, $\varepsilon$, $R$, $\rho _R$ and $\| b\| _\infty$.
By using Chebyshev's inequality, for $|x-z|\leq c_0$ calculate $E[t\wedge \tau]$ as 
\begin{align*}
E[t\wedge \tau]
& = \int _0^{c_0(-\log |x-z| )^{-1}} P(\tau \geq s) ds +\int _{c_0(-\log |x-z| )^{-1}}^t P(\tau \geq s) ds\\
& \leq E[ (c_0(-\log |x-z| )^{-1}) \wedge \tau] + t P(\tau \geq c_0(-\log |x-z| )^{-1} )\\
& \leq E[(c_0(-\log |x-z| )^{-1}) \wedge \tau] + \frac{t(-\log |x-z| )}{c_0} E[ \tau \wedge (c_0(-\log |x-z| )^{-1})] \\
&\leq \left( 1+ \frac{t(-\log |x-z| )}{c_0} \right) E[(c_0(-\log |x-z| )^{-1}) \wedge \tau].
\end{align*}
Thus, applying (\ref{eq:estprob1-9}) for $t=c_0(-\log |x-z| )^{-1}$ and choosing another small $\varepsilon$, we obtain (\ref{eq:estprob1-0}) for all and $x,z \in B(0;R/2)$ such that $|x-z| \leq c_0$.
The argument of the compactness enables us to remove the condition that $|x-z| \leq c_0$, and the desired assertion holds.
\end{proof}

\begin{lem}\label{lem:estprob1-2}
Additionally assume that $\int _0^1 r^{-1}\rho _R (r) dr <\infty$ for all $R>0$.
Then, for $R>0$ there exists a positive constant $C$ depending on $d$, $\Lambda$, $R$, $\rho _R$ and $\| b\| _\infty$ such that
\begin{equation}\label{eq:estprob1-2-0}
E[t\wedge \tau] \leq C(1+t^2)|x-z| \max\{ 1, -\log |x-z|\}
\end{equation}
for $t\in [0,T]$ and $x,z \in B(0;R/2)$.
\end{lem}

\begin{proof}
Let $R>0$ and $x,z \in B(0;R/2)$.
We define $\xi _t$, $\alpha _t$, $\beta _t$ as same as in the proof of Lemma \ref{lem:estprob1}.
By the same way as the proof of Lemma \ref{lem:estprob1}, we obtain the constants $\gamma _1$ and $\gamma _2$ satisfying  (\ref{eq:coup01}) and (\ref{eq:coup02}).
Let $\gamma := \gamma _1 \vee \gamma _2$ and define stopping times $\tau _n$ by $\tau _n := \inf \{ t>0; |X_t^x -Z_t^z| \leq 1/n \}$ for $n\in {\mathbb N}$, and
\begin{align*}
\tilde \tau &:= \tau \wedge \inf \left\{ t\in [0,\infty);\ \rho _R (|\xi _t|) > 2\Lambda ^{-1},\ X_t^x \not \in B(0;R) \ \mbox{or}\   Z_t^z \not \in B(0;R) \right\} \\
\tilde \tau _n &:= \tau _n \wedge \inf \left\{ t\in [0,\infty);\ \rho _R (|\xi _t|) > 2\Lambda ^{-1},\ X_t^x \not \in B(0;R) \ \mbox{or}\   Z_t^z \not \in B(0;R) \right\}
\end{align*}
for $n\in {\mathbb N}$.
Define a function $f$ on $[0, \infty)$ by
\[
f(\eta) := \int _0^\eta \exp \left( -\int _0^{\theta_1} \frac{2\gamma ^3 \rho _R(\theta _2)}{\theta _2} d\theta _2 \right) d\theta _1 ,\quad \eta \in [0,\infty ).
\]
Then, by It\^o's formula, (\ref{eq:xi}), (\ref{eq:coup01}) and (\ref{eq:coup02}), we have
\begin{align*}
&E[f(|\xi _{t\wedge \tilde \tau _n} |)] -f(|x-z|)\\
&= E\left[ \int _0^{t\wedge \tilde \tau _n} \left\{ \frac {f'(|\xi _s|)}{2|\xi _s|} \left( {\rm tr}( \alpha _s \alpha _s ^T) - \frac{|\alpha _s^T \xi _s|^2}{|\xi _s|^2} + \langle \xi _s , \beta _s \rangle \right) + \frac 12 f''(|\xi _s|) \frac{|\alpha _s^T \xi _s|^2}{|\xi _s|^2} \right\} ds \right] \\
&= E\left[ \int _0^{t\wedge \tilde \tau _n} \left\{ \frac {f'(|\xi _s|)}{2|\xi _s|} \left( {\rm tr}( \alpha _s \alpha _s ^T) - \frac{|\alpha _s^T \xi _s|^2}{|\xi _s|^2} + \langle \xi _s , \beta _s \rangle \right) - \frac {\gamma ^3 \rho _R(|\xi _s|)}{|\xi _s|} f'(|\xi _s|) \frac{|\alpha _s^T \xi _s|^2}{|\xi _s|^2} \right\} ds \right] \\
& \leq - E\left[ \int _0^{t\wedge \tilde \tau _n} f'(|\xi _s|) \frac {\gamma \rho _R(|\xi _s|)}{2|\xi _s|} ds \right] .
\end{align*}
Since 
\[
\inf _{\eta \in (0,\rho _R^{-1}(2\Lambda ^{-1}))} f'(\eta ) \frac {\gamma \rho _R(\eta)}{2\eta} >0,
\]
it holds that
\[
E\left[ \int _0^{t\wedge \tilde \tau _n} f'(|\xi _s|) \frac {\gamma \rho _R(|\xi _s|)}{2|\xi _s|} ds \right] \geq CE[t\wedge \tilde \tau _n]
\]
where $C$ is a positive constant depending on $d$, $R$, $\rho _R$, $\Lambda$ and $\| b\| _\infty$.
Hence,
\begin{equation}\label{eq:estprob1-2-7}
E[f(|\xi _{t\wedge \tilde \tau _n} |)] + CE[t\wedge \tilde \tau _n] \leq f(|x-z|) .
\end{equation}
Thus, by using the nonnegativity of $f$ and taking limit as $n\rightarrow \infty$, we obtain
\begin{equation}\label{eq:est-nhit2}
E\left[ t\wedge \tilde \tau \right] \leq C|x-z| \quad \mbox{for}\ t\in [0,\infty )
\end{equation}
where $C$ is a constant depending on $d$, $\Lambda$, $R$ and $\| b\| _\infty$.
By following the proof of Lemma \ref{lem:estprob1} with applying (\ref{eq:est-nhit2}) instead of (\ref{eq:est-nhit}), we have
\begin{equation}\label{eq:estprob1-2-9}
E[t\wedge \tau] \leq C(1+t)|x-z|
\end{equation}
for $x,z \in B(0;R/2)$ such that $|x-z| \leq c_0$, and $t\in [0, c_0(-\log |x-z| )^{-1} ]$ where $c_0$ and $C$ are positive constants depending on $d$, $\Lambda$, $\varepsilon$, $R$, $\rho _R$ and $\| b\| _\infty$.
Since, as in the proof of Lemma \ref{lem:estprob1}, it holds that
\[
E[t\wedge \tau] \leq \left( 1+ \frac{t(-\log |x-z| )}{c_0} \right) E[(c_0(-\log |x-z| )^{-1}) \wedge \tau],
\]
by applying (\ref{eq:estprob1-2-9}) with $t=c_0(-\log |x-z| )^{-1}$ we obtain the assertion.
\end{proof}

\begin{lem}\label{lem:estprob1-3}
If $\int _0^1 r^{-1}\rho (r) dr <\infty$, 
then there exists a positive constant $C$ depending on $d$, $\rho$, $\Lambda$ and $\| b\| _\infty$ such that
\begin{equation}\label{eq:estprob1-3-0}
E[t\wedge \tau] \leq C(1+t)|x-z|
\end{equation}
for $t\in [0,T]$ and $x,z \in {\mathbb R}^d$.
\end{lem}

\begin{proof}
We define $\xi _t$, $\alpha _t$, $\beta _t$ as same as in the proof of Lemma \ref{lem:estprob1}.
Similarly to the proof of Lemma \ref{lem:estprob1}, we obtain a positive constant $\gamma _1$ depending on $d$ and $\Lambda$ such that
\begin{equation}\label{eq:coup01-3}
\left| {\rm tr}( \alpha _t \alpha _t ^T) - \frac{|\alpha _t^T \xi _t|^2}{|\xi _t|^2} + \left\langle \xi _t, \beta _t \right\rangle\right| \leq \gamma _1 \rho (|\xi _t|) ,\ t\in [0,\tau),
\end{equation}
and (\ref{eq:coup02}).
Let $\gamma := \gamma _1 \vee \gamma _2$.
Note that (\ref{eq:coup01-3}) holds without the condition $X_t^x, Z_t^z \in B(0;R)$.

Since (\ref{eq:coup01-3}) holds without the condition $X_t^x, Z_t^z \in B(0;R)$, a similar argument to the proof of Lemma \ref{lem:estprob1-2} is available with respect to $\tilde \tau$ and $\tilde \tau _n$ defined by $\tau _n := \inf \{ t>0; |X_t^x -Z_t^z| \leq 1/n \}$ for $n\in {\mathbb N}$,
and
\begin{align*}
\tilde \tau &:= \tau \wedge \inf \left\{ t\in [0,\infty);\ \rho  (|\xi _t|) > 2\Lambda ^{-1} \right\} \\
\tilde \tau _n &:= \tau _n \wedge \inf \left\{ t\in [0,\infty);\ \rho (|\xi _t|) > 2\Lambda ^{-1} \right\}
\end{align*}
for $n\in {\mathbb N}$.
From this fact, we obtain (\ref{eq:estprob1-2-7}) with respect to $\tilde \tau _n$ defined in the present proof, and hence,
\begin{align}
E\left[ t\wedge \tilde \tau \right] &\leq C|x-z| \quad \mbox{for}\ t\in [0,T] \label{eq:est-nhit3-1}\\
E[f(|\xi _{t\wedge \tilde \tau} |)] &\leq C|x-z| \quad \mbox{for}\ t\in [0,T] \label{eq:est-nhit3-2}
\end{align}
where $f$ is the same function defined in the proof of Lemma \ref{lem:estprob1-2} and $C$ is a constant depending on $d$, $\Lambda$ and $\| b\| _\infty$.
The nonnegativity of $f$ and (\ref{eq:est-nhit3-2}) imply that for $t\in [0,T]$
\[
f\circ \rho ^{-1}(2\Lambda ^{-1}) P(\tilde \tau <\tau \wedge t) \leq E[f(|\xi _{t\wedge \tilde \tau } |)] \leq C|x-z|
\]
where $C$ is the constant appeared in (\ref{eq:est-nhit3-2}).
Hence, we have
\begin{equation}
P(\tilde \tau <\tau \wedge t) \leq C|x-z| \quad \mbox{for}\ t\in [0,\infty ) \label{eq:estprob1-3-1}
\end{equation}
where $C$ is a constant depending on $d$, $\Lambda$ and $\| b\| _\infty$.
On the other hand, for $t\in [0,T]$ it holds that
\begin{align*}
E[t\wedge \tau] &= \int _0^t P(\tau \geq s) ds\\
&\leq \int _0^t P( \tilde \tau \geq s) ds + \int _0^t P(\tilde \tau <\tau \wedge s) ds \\
&\leq E[t \wedge \tilde \tau] + t P(\tilde \tau <\tau \wedge t) .
\end{align*}
Therefore, we obtain the assertion from this inequality, (\ref{eq:est-nhit3-1}) and (\ref{eq:estprob1-3-1}).
\end{proof}

\begin{lem}\label{lem:estprob3}
For $p\in (1,\infty)$ it holds that
\begin{align*}
& \left| \int _{{\mathbb R}^d} f(y)p^X(0,x;T,y)E^{X_T^x = y}[T\wedge \tau] dy \right| \\
& \leq (s+T) s^{-1/p} E\left[ s\wedge \tau\right] ^{1/p} E\left[ \left( \int _{{\mathbb R}^d} |f(y)| p^X\left( s,X^x_s ;T,y\right) dy \right) ^{p^*}\right] ^{1/p^*}
\end{align*}
for $f\in C_b({\mathbb R}^d)$, $s \in (0,T)$, and $x,z\in {\mathbb R}^d$.
\end{lem}

\begin{proof}
It holds that
\begin{equation}\label{eq:estprob3-01}
E^{X^x_T=y}\left[ T\wedge \tau \right] = E^{X^x_T=y}\left[ (T\wedge \tau) {\mathbb I}_{[0,s]}(\tau) \right] + E^{X^x_T=y}\left[ (T\wedge \tau) {\mathbb I}_{(s,\infty)}(\tau) \right] .
\end{equation}
By (\ref{eq:conditional}) we have
\begin{align*}
&\left| \int _{{\mathbb R}^d} f(y) p^X(0,x;T,y) E^{X^x_T=y}\left[ (T\wedge \tau) {\mathbb I}_{[0,s]}(\tau) \right] dy \right| \\
&= \left| \int _{{\mathbb R}^d} f(y) E\left[ (T\wedge \tau) {\mathbb I}_{[0,s]}(\tau)\ p^X\left( s,X^x _s;T,y\right) \right] dy \right| \\
&\leq E\left[ (T\wedge \tau) {\mathbb I}_{[0,s]}(\tau) \int _{{\mathbb R}^d} |f(y)| p^X\left( s,X^x _s;T,y\right) dy \right].
\end{align*}
Thus, by the inequality
\[
E\left[ (t\wedge \tau)^{p}{\mathbb I}_{[0,s]}(\tau) \right] \leq s^{p-1}E\left[ s\wedge \tau \right] 
\]
and H\"older's inequality we have
\begin{equation}\label{eq:estprob3-02}\begin{array}{l}
\displaystyle \left| \int _{{\mathbb R}^d} f(y) p^X(0,x;T,y) E^{X^x_T=y}\left[ (T\wedge \tau) {\mathbb I}_{[0,s]}(\tau) \right] dy \right|\\
\displaystyle \leq s^{1/p^*} E\left[ s\wedge \tau \right] ^{1/p} E\left[ \left( \int _{{\mathbb R}^d} |f(y)| p^X\left( s,X^x_s ;T,y\right) dy \right) ^{p^*}\right] ^{1/p^*}.
\end{array}\end{equation}
On the other hand, by (\ref{eq:conditional}) we have
\begin{align*}
&\left| \int _{{\mathbb R}^d} f(y) p^X(0,x;T,y) E^{X^x_T=y}\left[ (T\wedge \tau) {\mathbb I}_{(s,\infty)}(\tau) \right] dy \right| \\
&\leq T \int _{{\mathbb R}^d} |f(y)| p^X(0,x;T,y) E^{X^x_T=y}\left[ {\mathbb I}_{(s,\infty)}(\tau) \right] dy \\
&= T \int _{{\mathbb R}^d} |f(y)| E\left[ {\mathbb I}_{(s,\infty)}(\tau) p^X\left( s,X^x _s;T,y\right) \right] dy \\
&= T E\left[ {\mathbb I}_{(s,\infty)}(\tau) \int _{{\mathbb R}^d} |f(y)| p^X\left( s,X^x _s;T,y\right) dy \right] \\
&\leq T P(\tau >s)^{1/p} E\left[ \left( \int _{{\mathbb R}^d} |f(y)| p^X\left( s,X^x_s ;T,y\right) dy \right) ^{p^*}\right] ^{1/{p^*}}\\
&\leq T P(s\wedge \tau \geq s)^{1/p} E\left[ \left( \int _{{\mathbb R}^d} |f(y)| p^X\left( s,X^x_s ;T,y\right) dy \right) ^{p^*}\right] ^{1/{p^*}} .
\end{align*}
Hence, by applying Chebyshev's inequality we have
\begin{equation}\label{eq:estprob3-03}\begin{array}{l}
\displaystyle \left| \int _{{\mathbb R}^d} f(y) p^X(0,x;T,y) E^{X^x_T=y}\left[ (T\wedge \tau) {\mathbb I}_{(s,\infty)}(\tau) \right] dy \right| \\
\displaystyle \leq s^{-1/p}TE\left[ s\wedge \tau \right] ^{1/p} E\left[ \left( \int _{{\mathbb R}^d} |f(y)| p^X\left( s,X^x_s ;T,y\right) dy \right) ^{p^*}\right] ^{1/p^*} .
\end{array}\end{equation}
Thus, we obtain the assertion by (\ref{eq:estprob3-01}), (\ref{eq:estprob3-02}) and (\ref{eq:estprob3-03}).
\end{proof}

Now we consider the estimate of the solution by the expectation of $\tau$.

\begin{prop}\label{prop2}
Let $p\in (1,\infty)$. Then, it holds that
\begin{align*}
&|u(T,x) - u(T,z)| \\
&\leq 2(s+T) (\| c\|_\infty + s^{-1}) s^{-1/p} e^{2\| c\|_\infty T}  E \left[ s \wedge \tau \right] ^{1/p} \\
&\quad \hspace{2cm} \times \max_{\eta = x,z} E\left[ \left( \int _{{\mathbb R}^d} |f(y)| p^X\left( s,X^\eta_s;T,y\right) dy \right)^{p^*}\right] ^{1/p^*}
\end{align*}
for $f\in C_b({\mathbb R}^d)$, $s \in (0,T)$, and $x,z \in {\mathbb R}^d$.
\end{prop}

\begin{proof}
Let $s \in (0,T)$, $x,y,z \in {\mathbb R}^d$ such that $x\neq z$, and $t_0 \in (s,T)$.
Recall that $X^z$ and $Z^z$ have the same law.
By (\ref{eq:conditional}) we have
\begin{align*}
&\left| E\left[ p^X({t_0},X^x_{t_0} ;T,y) \exp \left( \int _0^{t_0} c(T-u,X_u^x)du\right) ;\ \tau \leq s \right] \right.\\
& \quad \hspace{4cm} \left.- E\left[ p^X({t_0},Z^z_{t_0} ;T,y) \exp \left( \int _0^{t_0} c(T-u,Z_u^z)du\right) ;\ \tau \leq s \right] \right| \\
& \leq E\left[ \exp \left( \int _0^{t_0} c(T-u,Z_u^z)du\right) \left| p^X({t_0},X^x_{t_0} ;T,y) - p^X({t_0},Z^z_{t_0} ;T,y) \right|  ;\ \tau \leq s \right] \\
& \quad + E\left[ \left| \exp \left( \int _0^{\tau \wedge {t_0}} c(T-u,X_u^x)du\right) - \exp \left( \int _0^{\tau \wedge {t_0}} c(T-u,Z_u^z)du\right) \right| \right.\\
&\quad \hspace{5cm} \left. \times \exp \left( \int _{\tau \wedge {t_0}}^{t_0} c(T-u,Z_u^z)du\right) p^X({t_0},X^x_{t_0} ;T,y) ;\ \tau \leq s \right] \\
& \quad +E\left[ \exp \left( \int _0^{\tau \wedge {t_0}} c(T-u,X_u^x)du\right)\right.\\
&\quad \left. \phantom{\int _{\tau \wedge {t_0}}^{t_0}}  \times \left|\exp \left( \int _{\tau \wedge {t_0}}^{t_0} c(T-u,X_u^x)du\right) - \exp \left( \int _{\tau \wedge {t_0}}^{t_0} c(T-u,Z_u^z)du\right) \right| p^X({t_0},X^x_{t_0} ;T,y)  ;\ \tau \leq s \right] .
\end{align*}
Noting that
\[ 
X^x_u=Z^z_u\quad \mbox{for}\ u\geq \tau,
\]
we obtain
\begin{equation}\label{eq:2inital01}\begin{array}{l}
\displaystyle \left| p^X(0,x;T,y) E^{X^x_T =y}\left[ \exp \left( \int _0^{t_0} c(T-u,X_u^x)du\right) ;\ \tau \leq s \right] \right.\\
\displaystyle \quad \hspace{2cm} \left.- p^X(0,z;T,y) E^{X^x_T =y}\left[ \exp \left( \int _0^{t_0} c(T-u,X_u^z)du\right) ;\ \tau \leq s \right] \right| \\
\displaystyle \leq E\left[ \left| \exp \left( \int _0^{\tau \wedge {t_0}} c(T-u,X_u^x)du\right) - \exp \left( \int _0^{\tau \wedge {t_0}} c(T-u,Z_u^z)du\right) \right| \right.\\
\displaystyle \quad \hspace{3cm} \left. \times \exp \left( \int _{\tau \wedge {t_0}}^{t_0} c(T-u,Z_u^z)du\right) p^X({t_0},X^x_{t_0} ;T,y) ;\ \tau \leq s \right] .
\end{array}\end{equation}
By the triangle inequality, the mean-value theorem and (\ref{eq:conditional}), we obtain
\begin{align*}
& E\left[ \left| \exp \left( \int _0^{\tau \wedge {t_0}} c(T-u,X_u^x)du\right) - \exp \left( \int _0^{\tau \wedge {t_0}} c(T-u,Z_u^z)du\right) \right| \right.\\
&\quad \hspace{3cm} \left. \times \exp \left( \int _{\tau \wedge {t_0}}^{t_0} c(T-u,Z_u^z)du\right) p^X({t_0},X^x_{t_0} ;T,y) ;\ \tau \leq s \right] \\
& \leq e^{\| c\|_\infty T}E\left[ \left| \exp \left( \int _0^{\tau \wedge {t_0}} c(T-u,X_u^x)du\right) - 1 \right| p^X({t_0},X^x_{t_0} ;T,y) ;\ \tau \leq s \right] \\
&\quad + e^{\| c\|_\infty T} E\left[ \left| \exp \left( \int _0^{\tau \wedge {t_0}} c(T-u,Z_u^z)du\right) - 1 \right| p^X({t_0},X^x_{t_0} ;T,y) ;\ \tau \leq s \right] \\
& \leq 2\| c\|_\infty e^{2\| c\|_\infty T}E\left[ (T\wedge \tau) p^X({t_0},X^x_{t_0} ;T,y) ;\ \tau \leq s \right] \\
& \leq 2\| c\|_\infty e^{2\| c\|_\infty T}E\left[ (s\wedge \tau) p^X({t_0},X^x_{t_0} ;T,y) \right] \\
& \leq 2\| c\|_\infty e^{2\| c\|_\infty T}p^X(0,x;T,y) E^{X^x_t =y}\left[ s\wedge \tau \right] .
\end{align*}
Hence, by Lemma \ref{lem:estprob3} and (\ref{eq:2inital01}) we have
\begin{equation}\label{eq:2inital03} \begin{array}{l}
\displaystyle \left| \int _{{\mathbb R}^d} f(y) p^X(0,x;T,y) E^{X^x_T =y}\left[ \exp \left( \int _0^{t_0} c(T-u,X_u^x)du\right) ;\ \tau \leq s \right] dy\right.\\
\displaystyle \quad \hspace{1cm} \left.- \int _{{\mathbb R}^d} f(y) p^X(0,z;T,y) E^{X^x_T =y}\left[ \exp \left( \int _0^{t_0} c(T-u,X_u^z)du\right) ;\ \tau \leq s \right] dy \right| \\[3mm]
\displaystyle \leq 2\| c\|_\infty e^{2\| c\|_\infty T} (s+T)s^{-1/p}E[ s\wedge \tau ] ^{1/p} \\
\displaystyle \quad \hspace{4cm} \times E \left[ \left( \int _{{\mathbb R}^d} |f(y)| p^X\left( s,X^x_s ;T,y\right) dy \right) ^{p^*} \right]^{1/p^*}
\end{array}\end{equation} 
for $f\in C_b({\mathbb R}^d)$, $s,t_0\in (0,T)$ such that $s<t_0$, and $x,z \in {\mathbb R}^d$.

On the other hand, Chebyshev's inequality implies
\begin{align*}
E^{X_T^x=y}\left[ \exp \left( \int _0^{t_0} c(T-u,X_u^x)du\right) ;\ \tau \geq s \right] &\leq e^{\| c\|_\infty T} P^{X_T^x=y} \left( s\wedge \tau \geq s\right) \\
&\leq s^{-1}e^{\| c\|_\infty T} E^{X_T^x=y} \left[ s \wedge \tau\right].
\end{align*}
Hence, by Lemma \ref{lem:estprob3} we obtain
\begin{equation}\label{eq:2inital04}\begin{array}{l}
\displaystyle \left| \int _{{\mathbb R}^d} f(y) p^X(0,x;T,y) E^{X_T^x=y}\left[ \exp \left( \int _0^{t_0} c(T-u,X_u^x)du\right) ;\ \tau \geq s \right] dy \right| \\[3mm]
\displaystyle \leq (s+T) s^{-(1+1/p)} e^{\| c\|_\infty T} E[ s\wedge \tau ] ^{1/p}E \left[ \left( \int _{{\mathbb R}^d} |f(y)| p^X\left( s,X^x_s ;T,y\right) dy \right) ^{p^*} \right]^{1/p^*}
\end{array}\end{equation}
for $f\in C_b({\mathbb R}^d)$, $s,t_0\in (0,T)$ such that $s<t_0$, and $x,z \in {\mathbb R}^d$.
Similarly we have
\begin{equation}\label{eq:2inital04-2}\begin{array}{l}
\displaystyle \left| \int _{{\mathbb R}^d} f(y) p^X(0,z;T,y) E^{Z_T^z=y}\left[ \exp \left( \int _0^{t_0} c(T-u,Z_u^z)du\right) ;\ \tau \geq s \right] dy \right| \\[3mm]
\displaystyle \leq (s+T) s^{-(1+1/p)} e^{\| c\|_\infty T} E[ s\wedge \tau ] ^{1/p} E \left[ \left( \int _{{\mathbb R}^d} |f(y)| p^X\left( s,Z^z_s ;T,y\right) dy \right) ^{p^*} \right]^{1/p^*}
\end{array}\end{equation}
for $f\in C_b({\mathbb R}^d)$, $s,t_0\in (0,T)$ such that $s<t_0$, and $x,z \in {\mathbb R}^d$.

Thus, from (\ref{eq:2inital03}), (\ref{eq:2inital04}) and (\ref{eq:2inital04-2}), we obtain
\begin{align*}
&\left| \int _{{\mathbb R}^d} f(y) E\left[ p^X({t_0},X^x_{t_0} ;T,y) \exp \left( \int _0^{t_0} c(T-u,X_u^x)du\right)\right] dy \right.\\
&\quad \left. \hspace{2cm} - \int _{{\mathbb R}^d} f(y) E\left[ p^X({t_0},Z^z_{t_0} ;T,y) \exp \left( \int _0^{t_0} c(T-u,Z_u^z)du\right)\right] dy \right| \\
&\leq 2(s+T) (\| c\|_\infty + s^{-1}) s^{-1/p} e^{2\| c\|_\infty T} E[ s\wedge \tau ] ^{1/p} \\
&\quad \hspace{3cm} \times \max_{\eta = x,z} E \left[ \left( \int _{{\mathbb R}^d} |f(y)| p^X\left( s,X^x_s ;T,y\right) dy \right) ^{p^*} \right]^{1/p^*}
\end{align*}
for $f\in C_b({\mathbb R}^d)$, $s,t_0 \in (0,T)$ such that $s<t_0$, and $x,z \in {\mathbb R}^d$.
Therefore, since (\ref{eq:conditional}) and (\ref{eq:fundamental}) imply that for $\eta =x,z$
\begin{align*}
&\lim _{t_0 \uparrow T} \int _{{\mathbb R}^d} f(y) E\left[ p^X(t_0,X_{t_0}^\eta;T,y) \exp \left( \int_0^{t_0} c(T-u,X_u^\eta) du\right) \right] dy \\
&= \lim _{t_0 \uparrow T} \int _{{\mathbb R}^d} f(y) p^X(0,\eta;T,y) E^{X_T^\eta=y}\left[ \exp \left( \int_0^{t_0} c(T-u,X_u^\eta) du\right) \right] dy\\
&= \lim _{t_0 \uparrow T} E\left[ f(X^{\eta}(T)) \exp \left( \int_0^{t_0} c(T-u,X_u^\eta) du\right) \right] dy\\
&= E\left[ f(X^{\eta}(T)) \exp \left( \int_0^T c(T-u,X_u^\eta) du\right) \right] dy\\
&= \int _{{\mathbb R}^d} f(y) p^X(0,\eta;T,y) E^{X_T^x=y}\left[ \exp \left( \int_0^T c(T-u,X_u^\eta) du\right) \right] dy\\
&= \int _{{\mathbb R}^d} f(y) p(0,\eta;T,y) dy\\
&= u(T,\eta),
\end{align*}
we obtain the assertion.
\end{proof}

\section{The case of general $a$}\label{sec:general}

Let $p\in (1,\infty)$.
Let $a(t,x)=(a_{ij}(t,x))$ be a symmetric $d\times d$-matrix-valued bounded measurable function on $[0,\infty ) \times {\mathbb R}^d$ satisfying (\ref{ass:a1}) and that $a(t,\cdot )$ is continuous uniformly in $t$.
Choose $d\times d$-matrix-valued function $\sigma (t,x)$ such that $\sigma_{ij}(t,\cdot)$ is continuous uniformly in $t$, $a(t,x)= \sigma (t,x) \sigma (t,x)^T$ and (\ref{eq:sigma}) holds.
Let $\sigma ^{(n)}$ be a sequence whose components are smooth and $\sigma ^{(n)}(t,x)$ converges to that of $\sigma (t,x)$ for each $(t,x) \in [0,\infty) \times {\mathbb R}^d$.
Denote $\sigma ^{(n)}(t,x) \sigma ^{(n)}(t,x) ^T$ by $a^{(n)}(t,x)$.
Let $b^{(n)}(t,x)$ be a sequence of ${\mathbb R}^d$-valued smooth functions such that $\| b^{(n)}\| _\infty \leq \| b\| _\infty$ and $b^{(n)}(t,x)$ converges to $b(t,x)$ almost every $(t,x)$ with respect to the Lebesgue measure $dt\times dx$.
Consider the following parabolic partial differential equation
\begin{equation}\label{PDEn}
\left\{ \begin{array}{rl}
\displaystyle \frac{\partial}{\partial t} u^{(n)}(t,x) = & \displaystyle \frac 12 \sum _{i,j=1}^d a_{ij}^{(n)}(t,x)\frac{\partial ^2}{\partial x_i \partial x_j} u^{(n)}(t,x) +  \sum _{i=1}^d b_i^{(n)} (t,x) \frac{\partial }{\partial x_i} u^{(n)}(t,x) \\[3mm]
& \displaystyle + c(t,x) u^{(n)}(t,x) \\[3mm]
\displaystyle u^{(n)}(0,x) =&f(x).
\end{array}\right.
\end{equation}
Let $T>0$.
Similarly to Section \ref{sec:representation} we consider the following stochastic differential equation associated to (\ref{PDEn}):
\begin{equation}\label{SDEn}
\left\{ \begin{array}{rl}
\displaystyle dX_t^{(n),x} & \displaystyle = \sigma ^{(n)}(T-t,X_t^{(n),x})dB_t + b^{(n)}(T-t,X_t^{(n),x})dt,\quad t\in [0,T],\\
\displaystyle X_0^{(n),x}& \displaystyle =x .
\end{array}\right.
\end{equation}
Denote the transition probability density function of $X^{(n)}$ by $p^{X^{(n)}}(s,x;t,y)$.
Then, Theorem 11.3.4 of \cite{StVa} implies that
\begin{equation}\label{eq:conv1}
\lim _{n\rightarrow \infty} \int _{{\mathbb R}^d} |f(y)| p^{X^{(n)}}\left( s,\xi _n ;t,y\right) dy = \int _{{\mathbb R}^d} |f(y)| p^X\left( s,\xi ;t,dy\right) 
\end{equation}
for any $f\in C_b({\mathbb R}^d)$, $s,t\in (0,T]$ such that $s<t$, and $\{ \xi _n\} \subset {\mathbb R}^d$ such that $\lim _{n\rightarrow \infty}\xi _n = \xi$.
Besides, Theorem 11.3.4 of \cite{StVa} and Theorem 2.7 of Chapter I in \cite{IW} imply that for each $x\in {\mathbb R}^d$, there exist another probability space $(\tilde \Omega , \tilde {\mathscr F}, {\tilde P})$, a sequence of stochastic processes $\tilde X^{(n),x}$ and a stochastic process $\tilde X^x$ such that the law of $\tilde X^{(n),x}$ equals to that of $X^{(n),x}$, the law of $\tilde X^{x}$ equals to that of $X^x$, and $X^{(n),x}$ converges to $X^x$ on $C([0,T];{\mathbb R}^d)$ almost surely.
As the consequence of this fact, we obtain by (\ref{eq:uX})
\begin{equation}\label{eq:conv2}
\lim _{n\rightarrow \infty}u^{(n)}(T,x) = u(T,x)
\end{equation}
for each $x\in {\mathbb R}^d$, and also we obtain by (\ref{eq:conv1})
\begin{equation}\label{eq:conv3}\begin{array}{l}
\displaystyle \lim _{n\rightarrow \infty} E\left[ \left( \int _{{\mathbb R}^d} |f(y)| p^{X^{(n)}}\left( s,X^{(n),x}_s ;t,y\right) dy \right) ^{p^*}\right]\\
\displaystyle = E\left[ \left( \int _{{\mathbb R}^d} |f(y)| p^X\left( s,X^{x}_s ;t,dy\right) \right) ^{p^*}\right]
\end{array}\end{equation}
for any $f\in C_b({\mathbb R}^d)$, $s,t\in (0,T]$ such that $s<t$, and $x \in {\mathbb R}^d$.

On the other hand, Lemma \ref{lem:estprob1} and Proposition \ref{prop2} implies that for $R>0$ and sufficiently small $\varepsilon >0$, there exists a positive constant $C$ depending on $d$, $\Lambda$, $\varepsilon$, $R$, $\rho _R$, $\| b\| _\infty$ and $\| c\| _\infty$
\begin{align*}
&|u^{(n)}(T,x) - u^{(n)}(T,z)| \\
&\leq Cs^{-(1+1/p)}T e^{CT} |x-z|^{(1-\varepsilon )/p} E\left[ \left( \int _{{\mathbb R}^d} |f(y)| p^{X^{(n)}}\left( s,X^{(n),x}_s;T,y\right) dy \right)^{p^*}\right] ^{1/p^*}
\end{align*}
for $f\in C_b({\mathbb R}^d)$, $s\in (0,T)$, and $x,z \in B(0;R/2)$.
Hence, by this inequality, (\ref{eq:conv2}) and (\ref{eq:conv3}), \ref{thm:main1} of Theorem \ref{thm:main} is obtained for $p\in (1,\infty)$.
The case that $p=1$ follows by taking limit as $p\downarrow 1$.

We have \ref{thm:main2} and \ref{thm:main3} of Theorem \ref{thm:main} by applying Lemmas \ref{lem:estprob1-2} and \ref{lem:estprob1-3} instead of Lemma \ref{lem:estprob1}, respectively.

\section*{Appendix}

\renewcommand{\thethm}{A.\arabic{thm}}

\begin{prop}\label{prop:appendix1}
Let $({\mathscr F}_t)$ be a filtration and $(M_t)$ be a continuous square integrable martingale such that $M_0=0$ almost surely.
If there exists non-random processes $(\alpha _t)$ and $(\beta _t)$ such that $0 \leq \alpha _t \leq \langle M\rangle _t \leq \beta _t$ for $t\in [0,\infty )$ almost surely, then for $x\geq 0$
\begin{eqnarray*}
\sqrt{\frac 2\pi} \int _{x{\alpha _t}^{-1/2}}^\infty e^{-\xi ^2/2} d\xi \leq P\left( \sup _{s\in [0,t]}M_s \geq x\right) \leq \sqrt{\frac 2\pi} \int _{x{\beta _t}^{-1/2}}^\infty e^{-\xi ^2/2} d\xi ,\\
\sqrt{\frac 2\pi} \int _0 ^{x{\beta _t}^{-1/2}} e^{-\xi ^2/2} d\xi \leq P\left( \sup _{s\in [0,t]}M_s \leq x\right) \leq \sqrt{\frac 2\pi} \int _0 ^{x{\alpha _t}^{-1/2}} e^{-\xi ^2/2} d\xi .
\end{eqnarray*}
In particular, if $c_{1}t \leq \langle M\rangle _t \leq c_{2}t$ for $t\in [0,\infty )$ almost surely with some positive constants $c_1$ and $c_2$, then for $t>0$ and $x>0$
\begin{align*}
P\left( \sup _{s\in [0,t]}M_s \geq x\right) &\leq \sqrt{\frac {2c_2t}{\pi x^2}} \exp\left(-\frac{x^2}{2c_2 t}\right) ,\\
P\left( \sup _{s\in [0,t]}M_s \leq x\right) &\leq \sqrt{\frac 2{c_1 \pi t}} x.
\end{align*}
\end{prop}

\begin{proof}
By Theorem 7.2 of Chapter II in \cite{IW}, there exists a Brownian motion $(B_t)$ satisfying that $M_t= B_{\langle M\rangle _t}$ for $t\in [0,\infty )$ almost surely (if necessary, extend the probability space).
Hence, 
\[
P\left( \sup _{s\in [0,t]}M_s \geq x\right) = P\left( \sup _{s\in [0,\langle M\rangle _t ]}B_s \geq x\right) .
\]
By the assumption we obtain
\[
P\left( \sup _{s\in [0,\alpha _t]}B_s \geq x\right) \leq P\left( \sup _{s\in [0,t]}M_s \geq x\right)  \leq P\left( \sup _{s\in [0,\beta _t ]}B_s \geq x\right) .
\]
On the other hand, it is known that for $t\in (0,\infty)$ and $x\geq 0$
\[
P\left( \sup _{s\in [0,t]}B_s \geq x\right) = \sqrt{\frac 2\pi} \int _{xt^{-1/2}}^\infty e^{-\xi ^2/2} d\xi
\]
(see e.g. Section 2.6.A of Chapter 2 in \cite{KaSh}).
Thus, we have the first assertion.
The second one is obtained similarly.
\end{proof}

\section*{Acknowledgment}
The author is grateful to Professor Shigeo Kusuoka for helpful comments.
The author also thank to the anonymous referee for helpful comments.
This work was supported by JSPS KAKENHI Grant number 25800054.

\bibliographystyle{plain}
\bibliography{Sei010.bib}

\end{document}